\documentclass[11pt]{article}
\usepackage{amssymb}
\usepackage{latexsym}
\usepackage{graphicx}

\newtheorem{theorem}{Theorem}

\newtheorem{corollary}[theorem]{Corollary}

\newtheorem{lemma}[theorem]{Lemma}

\newtheorem{observation}[theorem]{Observation}

\newtheorem{proposition}[theorem]{Proposition}

\newenvironment{proof}[1][Proof]{\noindent\textbf{#1.} }{\ \rule{0.5em}{0.5em}}
\def\lab(#1)#2{\put(#1){\makebox(0,0)[c]{#2}}}

\begin{document}

\title{\bf Generalised $k$-Steiner Tree Problems in Normed Planes}

\author{Marcus N.~Brazil, Charl J.~Ras, Konrad J.~Swanepoel, Doreen A.~Thomas}

\date{}
\maketitle

\begin{abstract}
The $1$-Steiner tree problem, the problem of constructing a Steiner minimum tree containing at most one Steiner point, has been solved in the
Euclidean plane by Georgakopoulos and Papadimitriou using plane subdivisions called oriented Dirichlet cell partitions. Their algorithm produces
an optimal solution within $O(n^2)$ time. In this paper we generalise their approach in order to solve the $k$-Steiner tree problem,
in which the Steiner minimum tree may contain up to $k$ Steiner points for a given constant $k$. We also extend their approach further to
encompass other normed planes, and to solve a much wider class of problems, including the $k$-bottleneck Steiner tree problem and other
generalised $k$-Steiner tree problems. We show that, for any fixed $k$, such problems can be solved in $O(n^{2k})$ time.

\medskip

\noindent\textit{Keywords:}  $k$-Steiner tree; bottleneck Steiner problem; network optimisation; polynomial time algorithm
\end{abstract}

\section{Introduction}

The \textit{geometric Steiner tree problem}, which asks for a network with minimum total edge length interconnecting a given set of points (called
\emph{terminals}), is a well known variant of the \textit{spanning tree problem}. The Steiner tree problem can be viewed as belonging
to a family of problems where the aim is to construct a network $T$ interconnecting a given set of $n$ terminals, with the following properties:
\begin{enumerate}
    \item $T$ may contain nodes (\textit{Steiner points}) other than the given terminals;
    \item $T$ minimises a given cost function;
    \item the given cost function guarantees that $T$ can be assumed to be a minimum spanning tree on its nodes (for a given metric).
\end{enumerate}

In addition to the Steiner tree problem under various metrics, this family also includes the \textit{power-$p$ Steiner tree problem}, where the
cost of each edge of the network is the $p$-th power of its length, and the \textit{bottleneck Steiner tree problem}, where the cost of the
network is the length of its longest edge and there is a bound on the number of extra nodes in the network. Both of these variants on the Steiner tree problem have numerous applications, particularly in the design of wireless communication networks such as sensor networks.

In contrast to the \textit{graph Steiner problem}, in the geometric versions of the problem the Steiner points may potentially be located at any of an infinite number of points in a given space. This makes the geometric Steiner tree problems, and the methods of finding optimal solutions, fundamentally different from their purely combinatorial analogues. In many versions of the geometric problem it is not even immediately clear how an optimal solution can be calculated.

Although the spanning tree problem is polynomially solvable, being solvable in $O(n^2)$ time in a general metric and in $O(n\log n)$ time in the
Euclidean plane, see \cite{shamos}, the problems in this more general family are mostly $N\! P$-hard, even in the plane. This essentially
stems from the fact that as the number of extra nodes that can be added to the network increases there is an exponential explosion in the number
of different topologies that need to be considered. A natural way of controlling this increase in complexity is to bound the number of extra
nodes, in other words, replace Property~1 above by the following:
\begin{enumerate}
    \item[1a.] $T$ may contain up to $k$ nodes other than the given terminals, where $k$ is a constant positive integer.
\end{enumerate}
We refer to problems in this modified class as  \textit{generalised $k$-Steiner tree problems}. One of the seminal papers on this topic, for the $1$-Steiner tree problem in the Euclidean plane, is a paper by Georgakopoulos and Papadimitriou
\cite{georg}, where an $O(n^2)$-time solution is given. They conclude their paper with a tantalising comment relating to their
unsuccessful attempts at generalising their methodology to the $k$-Steiner tree problem, even for $k=2$. This comment has been a
motivating factor for the current paper. Also, in their paper the proofs of some of the primary results and sufficient details are omitted;
another aim of this paper is to add some rigour to the more fundamental of these results.

It should be noted that solutions to $k$-Steiner problems are fundamentally different to the analogous classical Steiner problems where the number of Steiner points is not bounded. A case in point is the package of exact algorithms called GeoSteiner of Warme, Winter, Zachariasen (see \cite{warme} for one of the companion papers) for constructing optimal Euclidean and rectilinear Steiner trees. For a variety of reasons these algorithms cannot simply be adapted to solve the respective $k$-Steiner problems, whilst maintaining efficiency. A fundamental obstacle in the Euclidean plane is the fact that the degrees of Steiner points can be $4$ when $k$ is bounded, so many of the nice geometric properties utilised in GeoSteiner are lost.

We may restate the goal of the Georgakopoulos and Papadimitriou paper as follows: find the point in the Euclidean plane which, if added to a given set of points, will result in the shortest possible spanning tree. The authors observed that one can significantly reduce the time complexity of an algorithmic
solution to the problem by first constructing a special partition. Given a set $X$ of $n$ terminals in the Euclidean plane, it is possible to
partition the plane into $O(n^2)$ regions such that if any new point $\mathbf{s}$ is embedded in the plane within one of these
regions, say $R$, then any minimum spanning tree $T$ on $X \cup \{\mathbf{s}\}$ will have the following property: the neighbours of $\mathbf{s}$
in $T$ will belong to some subset of a set $C_X(R)$ containing at most six points from $X$, where $C_X(R)$ is fixed for the given region  $R$.
This useful partition is referred to as the \textit{overlayed oriented Dirichlet cell partition}, or OODC partition. Their algorithm takes as
input the set $X$ of terminals and then starts by calculating the OODC partition, the set $C_X(R)$ for every $R$, and a minimum spanning tree
$T^{\,\prime}$ on $X$. All this, as well as a preprocessing step on $T^{\,\prime}$, is done within a time of $O(n^2)$. For each region
$R$, the algorithm then iterates through all subsets $S_0$ of $C_X(R)$ and calculates $\mathbf{s}$, the Steiner point of the nodes in $S_0$ (this
takes constant time for each $S_0$). The algorithm then updates $T^{\,\prime}$ (which can also be done in constant time because of the
preprocessing step) to include $\mathbf{s}$. The cheapest tree is selected at the end as an optimal solution. A naive algorithm for the
$1$-Steiner tree problem would attain a complexity of $O(n^6\cdot\,n\log n)$ since it would have to iterate through \textit{all} subsets
of up to six terminals and then calculate the optimal position of the Steiner point and the corresponding minimum spanning tree for each of
these subsets.

The powerful simplifying properties of the OODC would clearly be advantageous in a generalisation of the algorithm to arbitrary normed planes.
However, the construction of the OODC partition given in \cite{georg} is valid for the Euclidean plane only. We will provide a new method, based
on abstract Voronoi diagrams, for constructing the partition for terminals embedded in an arbitrary normed plane. The construction is based on
theoretical results, but we also define a class of norms for which the algorithm would be practically implementable. Once the partition has been
found, our algorithm calculates the optimal positions of all $k$ Steiner points simultaneously. Of course, these positions will depend not only
on the neighbours of the Steiner points, but also on the cost function of the given generalised Steiner tree problem. Since, at the start of this
step, the neighbours of the Steiner points are fixed but the Steiner points are free, this subproblem is a generalised version of the well-known
\textit{fixed topology Steiner tree problem} (discussed in Section \ref{topology3}). Since $k$ is constant we assume that this step can be done in constant time. A novel method for
updating a minimum spanning tree is then utilised to calculate a potential solution for every choice of coordinates of the Steiner points. Once
again, a cheapest tree is selected as the optimal solution. The total time complexity turns out to be $O(n^{2k})$ when constant factors are excluded.

There are a number of authors who have looked at adapting the solution to the $1$-Steiner tree problem in \cite{georg} to other $\ell_p$ norms.
Kahng and Robins in \cite{kahng} do this for the rectilinear plane, however, their paper only uses the solution as a step in a heuristic
algorithm for the rectilinear Steiner tree problem, and not much attention is devoted to the solution of the $1$-Steiner tree problem itself.
Griffith et al. \cite{griffith} expand on this heuristic idea in the rectilinear plane. They provide a simple procedure (though without
proof) for updating a minimum spanning tree when a new node is introduced. Lin et al. \cite{lin} in turn adapt the approach presented by Kahn and Robins to
the norm induced by a regular hexagon. Recent papers \cite{bae2,bae1} by Bae et al. provide the first exact algorithms for solving the \textit{bottleneck} $k$-Steiner tree problem in the $\ell_p$ planes. The complexity of their algorithms are $O(n\log^2n)$ when $p=1$ and $O(n^k+n\log n)$ for all other finite $p$. However, the methods they use are based on farthest colour Voronoi diagrams and therefore cannot be utilised for any other cost functions. Besides these authors we are not aware of any significant study into the properties and construction of optimal geometric $k$-Steiner trees. Although the classical Steiner tree problem (where the number of Steiner points is not bounded) has been considered in a multitude of norms and under many cost functions, these results are mostly irrelevant to the $k$-Steiner problem.

Section \ref{prelim1} provides some preliminary definitions. Our algorithm for solving the generalised $k$-Steiner tree problem in normed planes has three primary phases. The first phase constructs a set of \textit{feasible internal topologies}. Each feasible internal topology is a forest with leaves only from the set $X$ of terminals, and interior nodes only from the set $S$ of Steiner points. At this stage the nodes of $S$ are not yet located in the plane. By utilising OODC partitions, as discussed in Section \ref{ODC2}, we are able to significantly reduce the total number of feasible internal topologies. In Section \ref{ODC2} we also present three restrictions on the given normed plane that allows the construction of the OODC partition to be implemented in practice. The next phase of our algorithm consists of finding the optimal locations in the plane of the nodes of $S$ for each feasible internal topology. This is known in the literature as the \textit{fixed topology Steiner tree problem}, and its solution depends crucially on the cost function $\alpha$ {and on the given normed plane}. We briefly discuss this phase of the algorithm again in Section \ref{topology3}. The final phase is to add each feasible internal topology (with Steiner points optimally located) to a minimum spanning tree on $X$. The union of these two graphs produces cycles and thus a method is needed for deleting the appropriate edges from the union until an optimal final tree is attained. This so called \textit{minimum spanning tree update} method is the topic of Section \ref{MST4}. In Section \ref{algorithm5} we present our main algorithm and then prove its correctness and verify its time complexity.

\section{Preliminaries}\label{prelim1}

We begin by formalising the definition of a generalised $k$-Steiner tree problem, sketched in the introduction. Throughout this paper we use the symbols $E(G)$ and $V(G)$ for the edge-set and node-set respectively of a graph $G$. We also use the notation
$G=\langle V(G),E(G)\rangle$. Let $k^\prime>0$ be given. Let $\|\cdot\|$ be a given norm on $\mathbb{R}^2$, that is, a function $\|\cdot\|:\mathbb{R}^2\rightarrow\mathbb{R}$ that satisfies $\|\mathbf{x}\|\geq 0$ for all $\mathbf{x}\in\mathbb{R}^2$, $\|\mathbf{x}\|=0$ if and only if $\mathbf{x}=0$, $\|\lambda\mathbf{x}\|=|\lambda\||\mathbf{x}\|$ for $\lambda\in \mathbb{R}$, and $\|\mathbf{x}+\mathbf{y}\|\leq \|\mathbf{x}\|+\|\mathbf{y}\|$ for all $\mathbf{x},\mathbf{y}\in\mathbb{R}^2$. The \textit{unit ball} $B=\{\mathbf{x}:\|\mathbf{x}\|\leq 1\}$ is a centrally symmetric convex set.

Given a set $P=\{ p_1, \ldots , p_{n+k^\prime}\}$,  let  $\{\mathcal{T}\}$ represent the set of all spanning trees for the elements of $P$. For
each $\mathcal{T}$ there is a corresponding set of edges $E(\mathcal{T})= \{ e_1, \ldots , e_{n+k^\prime-1}\}$ (with every $e_i \in P\times P$).
Let $X=\{\mathbf{x}_1, \ldots , \mathbf{x}_{n}\}$, with $\mathbf{x}_i \in \mathbb{R}^2$, represent an embedding of the set $\{ p_1, \ldots , p_{n}\}$ in
$\mathbb{R}^2$ (where $\mathbf{x}_i$ is an embedding of the corresponding $p_i$). We can think of $\mathcal{T}$ as representing the topology of
a tree network interconnecting $X$ and using $k^\prime$ extra nodes, and we can equate the edges $E(\mathcal{T})$ with the arcs of such a
network. For a fixed embedding of this network we let $S=\{\mathbf{x}_{n+1}, \ldots , \mathbf{x}_{n+k^\prime}\}$, with $\mathbf{x}_i \in \mathbb{R}^2$, be the
locations of the extra nodes corresponding to $\{ p_{n+1}, \ldots , p_{n+k^\prime}\}$. We refer to $X$ as the set of \emph{terminals} and $S$ as the set of \emph{Steiner points} of the network. Now let $\mathbf{e}_{\mathcal{T}, X,S}= ( \| e_1\| ,
\ldots , \| e_{n+k^\prime-1}\| )$; i.e., the components of $\mathbf{e}_{\mathcal{T}, X,S}$ are the edge lengths of such a network, for a given
tree topology and a given set of embedded nodes. Such a vector is well defined up to the order of its components.

Let $\alpha:\mathbb{R}^{n+k^\prime-1}_+ \rightarrow \mathbb{R}$ be a symmetric function (i.e., independent of the order of the components of the
vector on which it acts). We think of $\alpha$ as a cost function on a tree network. In other words, $\alpha(\mathbf{e}_{\mathcal{T}, X,S})$ is
the cost of the network with topology $\mathcal{T}$ and nodes $X$ and $S$, and $\displaystyle
\min_{\mathcal{T},S}\alpha(\mathbf{e}_{\mathcal{T}, X,S})$ is the minimum cost (with respect to $\alpha$) of any tree interconnecting the nodes
$X$ and $k^\prime$ other points. Hence, for the power-$p$ Steiner tree problem we define $$\alpha(\mathbf{e}_{\mathcal{T},
X,S})=\alpha_p(\mathbf{e}_{\mathcal{T}, X,S}):= \sum_{i=1}^{n+k^\prime-1}\| e_i\|^p;$$ whereas, for the bottleneck problem (where
the cost of the network is the cost of the longest edge) we have
$$\alpha(\mathbf{e}_{\mathcal{T}, X,S})=\alpha_\infty(\mathbf{e}_{\mathcal{T}, X,S}):= \max_{i=1,\ldots, n+k'-1}\| e_i\|.$$

We say that such a symmetric function $\alpha$ is \emph{$\ell_1$-optimisable} if and only if there  exist $\mathcal{T}^*$ and $S^*$ such that
$\displaystyle \alpha(\mathbf{e}_{\mathcal{T}^*, X,S^*})= \min_{\mathcal{T},S}\alpha(\mathbf{e}_{\mathcal{T}, X,S})$ and $\displaystyle
\alpha_1(\mathbf{e}_{\mathcal{T}^*, X,S^*})= \min_{\mathcal{T}}\alpha_1(\mathbf{e}_{\mathcal{T}, X,S^*})$. In other words, $\alpha$ is
\emph{$\ell_1$-optimisable} if for any given $X$ there exists a tree $T$ interconnecting $X$, with minimum cost with respect to $\alpha$, that
is a minimum spanning tree on its \textit{complete set} of nodes. We denote the cost of such a tree by $\| T \|_\alpha$. It is
easy to show that $\alpha_p$, for $p>0$, and $\alpha_\infty$ are $\ell_1$-optimisable.

\noindent \textbf{Definition.} For any constant positive integer $k$, a \emph{generalised k-Steiner tree problem} is defined to be any problem of
the following form:
\begin{quote}\begin{description}
    \item[\textbf{Given}] A set $X$ of $n$ points in $\mathbb{R}^2$, a norm  $\|\cdot\|$, and a symmetric $\ell_1$-optimisable function
    $\alpha$.%:\mathbb{R}^{n+k^\prime-1}_+ \rightarrow \mathbb{R}$.
    \item[\textbf{Find}] A set $S$ of $k^\prime\leq k$ points in $\mathbb{R}^2$, and a spanning tree $T$ on $X\cup S$ with topology
    $\mathcal{T}$ such that
    $\| T \|_\alpha =\alpha(\mathbf{e}_{\mathcal{T}, X,S})=
    {\displaystyle \min_{\mathcal{T}^\prime,S^\prime}\alpha(\mathbf{e}_{\mathcal{T}^\prime, X,S^\prime})}$.
\end{description} \end{quote}

We refer to $T$ as a \textit{generalised $k$-Steiner minimum tree}.  The next lemma is an extension of the Swapping
Algorithm, found in \cite{lee}. It follows from the matroid properties of minimum spanning trees.

\begin{lemma}\label{MSTswap}Let $T_0$ be a minimum spanning tree on the terminal set $X$, and let $T_1$ be any spanning tree for $X$.
We can transform $T_1$ to $T_0$ by a series of edge swaps, where each swap involves replacing an edge $e_i \in E(T_1)$ by $e_j
\in E(T_0)$ such that $\|e_i\| \geq \|e_j\|$.
\end{lemma}

The corollary below shows that $T$ is equivalent in cost to any  minimum spanning tree on $X\cup S$.

{
\begin{corollary}\label{MST-coroll}Every minimum spanning tree on $X\cup S$ is a generalised $k$-Steiner minimum tree on $X$.
\end{corollary}
}
{
\begin{proof}
By the $\ell_1$-optimisability of $\alpha$ we may assume that $T$ is a minimum spanning tree on $X\cup S$. Let $T^{\,\prime}$ be any other minimum spanning tree $X\cup S$. By Lemma~\ref{MSTswap}, we can transform $T^{\,\prime}$ to $T$ by a series of edge swaps, each
of which replaces an edge with another of the same length. By the symmetry of $\alpha$ each such edge swap does not increase $\| T^{\,\prime}
\|_{\alpha}$.
 \end{proof}
}

Throughout this paper we  perform various constructions involving the unit ball $B$ for the given norm $\|\cdot\|$,
for instance calculating the intersections of two unit balls. Our main interest in this paper is in the computational nature, specifically the
time complexity, of a solution to any instance of the generalised $k$-Steiner tree problem. In order to find efficient algorithms, we need to perform these unit ball operations to within any fixed precision in constant time. We therefore restrict the norm
$\|\cdot\|$ so that its unit ball is always simple enough to perform these operations. We will provide more detail regarding
these restrictions in the next section. For similar  computational reasons we will also be placing a restriction on $\alpha$, and this is
discussed in Section \ref{topology3}.

\section{The Overlayed Oriented Dirichlet Cell Partition}\label{ODC2}
Let a norm $\| \cdot \|$ on $\mathbb{R}^2$ be given with corresponding unit ball $B$. Our aim in this section is to describe the
construction of the oriented Dirichlet cell (ODC) partition for any set $X$ of $n$ terminals embedded in this normed plane, and to show that it
can be constructed within a time of $O(n\log n)$. We also show that, with a time complexity of $O(n^2)$, multiple ODC
partitions can be overlayed. This final partition is the aforementioned \textit{overlayed ODC partition} (OODC partition), and is a core
component of our algorithm.

Georgakopoulos and Papadimitriou \cite{georg} allude to a simple method of constructing an ODC partition for terminals embedded in the Euclidean
plane. Unfortunately this method does not work for arbitrary normed planes. We circumvent this problem by defining a type of abstract Voronoi
diagram that is equivalent to the ODC partition, and then showing that this Voronoi diagram can be calculated in the required time.

We now state the first of three restrictions on $B$. We defer a discussion of these restrictions (including the description of a class of norms
that satisfy all of them) to the end of the section. Let the boundary of $B$ be denoted by $\mathrm{bd}(B)$.

\noindent\textbf{Restriction 1} \textit{The intersection points of any two translated copies of $\mathrm{bd}(B)$, and the intersection points of
any straight line and $\mathrm{bd}(B)$, can be calculated to within any fixed precision in constant time.}

\begin{lemma}\label{yLemma}There exist six points $\{\mathbf{y}_i:i=0,...,5\}$ on $\mathrm{bd}(B)$ such that for any pair of rotationally
consecutive ones, say $\mathbf{y}_i,\mathbf{y}_j$, we have $\| \mathbf{y}_i-\mathbf{y}_j \| = 1$. Moreover, these six points are
constructible.
\end{lemma}
\begin{proof}The standard ruler and compass construction of the hexagon will produce these six points, where $B$ plays the role of the circle in
the construction. Given any point $\mathbf{y}_5$ on $\mathrm{bd}(B)$ construct a translation of $\mathrm{bd}(B)$ centered around $\mathbf{y}_5$.
Let $\mathbf{y}_0$ be the first intersection point of the two boundaries as we traverse the boundary of the original ball anticlockwise from
$\mathbf{y}_5$. Let $\mathbf{y}_1 = \mathbf{y}_0- \mathbf{y}_5$. Note that this point also lies on $\mathrm{bd}(B)$ and that
$\mathbf{y}_0\mathbf{y}_1\mathbf{o}\mathbf{y}_5$ is a parallelogram. The remaining three points are constructed using the central symmetry of
$B$. See Figure \ref{figureHex} for an example where $B$ is a tilted ellipse.  The distance properties of the lemma follow easily by
construction.
 \end{proof}

\begin{figure}[htb]
\begin{center}

\includegraphics[scale=0.5]{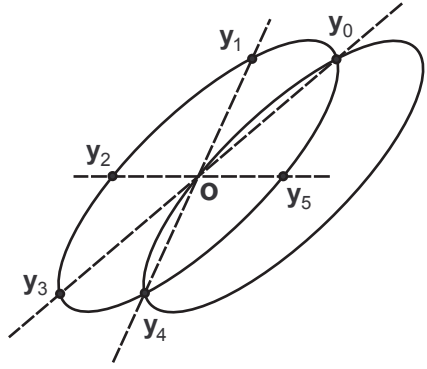}

\end{center}
\caption{The standard hexagon construction\label{figureHex}}
\end{figure}

For any two directions  $\phi_i$ and $\phi_j$ in the plane $K(\mathbf{y},\phi_i,\phi_j)$ denotes the cone defined to be the set consisting of all rays emanating from $\mathbf{y}$ in direction $\phi$, for $\phi_i\leq \phi\leq \phi_j$. For each $\mathbf{y}_i$ from Lemma \ref{yLemma} let $\theta_i$ be the
direction of the ray $\overrightarrow{\mathbf{o}\mathbf{y}_i}$, where $\mathbf{o}$ is the center of $B$. We assume that the $\{\theta_i\}$ are
ordered in an anticlockwise manner, and two consecutive directions will be denoted by $\theta_i$ and $\theta_{i+1}$ (i.e. the $\mathrm{mod}\ 6$
notation will be omitted). As another example we show, in Figure \ref{figureRect}, the six directions produced when $B$ is the unit ball of the
rectilinear plane.

\begin{figure}[htb]
\begin{center}

\includegraphics[scale=0.5]{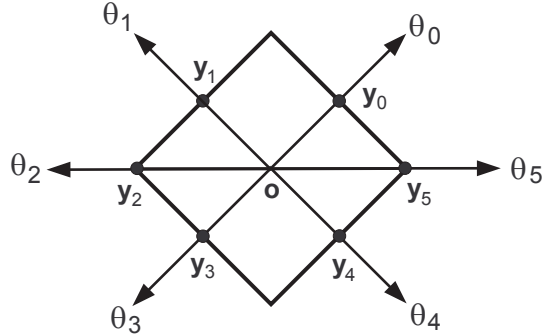}

\end{center}
\caption{The unit ball of the rectilinear plane and corresponding $\{\theta_i\}$\label{figureRect}}
\end{figure}

\begin{lemma}\label{tilemma}Let $\mathbf{x}\in B$, $\mathbf{x}\neq \mathbf{o}$ and $\mathbf{a}$ and $\mathbf{b}$ points on the boundary of $B$ such
that the segments $\mathbf{ao}$ and $\mathbf{bx}$ intersect in a point $\mathbf{p}$. Then $\|\mathbf{a}-\mathbf{x}\| \leq
\|\mathbf{b}-\mathbf{x}\|$.
\end{lemma}
\begin{proof}Applying the Triangle Inequality to $\triangle \mathbf{pax}$ and $\triangle \mathbf{pbo}$, we obtain
$\|\mathbf{a}-\mathbf{x}\| + \|\mathbf{b}\| \leq \|\mathbf{b}-\mathbf{x}\| +
\|\mathbf{a}\| $ from which the lemma follows; see Figure \ref{fig}.
 \end{proof}

\begin{figure}[htb]
\begin{center}

\includegraphics[scale=0.5]{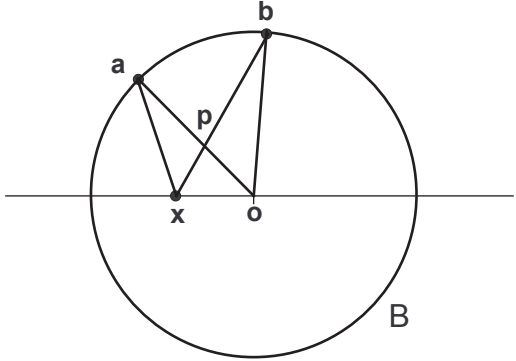}

\end{center}
\caption{Illustration of the proof of Lemma \ref{tilemma}\label{fig}}
\end{figure}

\begin{lemma}\label{cony}Let $\mathbf{y}$ be any point in the plane. Then there exists a minimum spanning tree $T$ on $X\cup \{\mathbf{y}\}$
with the following property: for each $i=0,...,5$ there is at most one point of $X$ adjacent to $\mathbf{y}$ in $T$ and lying within
$K(\mathbf{y}, \theta_i,\theta_{i+1})$, and this point is a closest terminal to $\mathbf{y}$ in the cone.
\end{lemma}

\begin{proof}  Let $T^{\,\prime}$ be a minimum spanning tree on $X\cup\{\mathbf{y}\}$. Let $\mathbf{x}_0\in X$ be a terminal in
$K(\mathbf{y}, \theta_i,\theta_{i+1})$ that is closest to $\mathbf{y}$, and suppose that $(\mathbf{y},\mathbf{x}_1)\in E(T^{\,\prime})$ where
$\mathbf{x}_1\in X$ is any other terminal in $K(\mathbf{y}, \theta_i,\theta_{i+1})$. We show that we can replace the edge
$(\mathbf{y},\mathbf{x}_1)$ in $T^{\,\prime}$ by either $(\mathbf{y},\mathbf{x}_0)$ or $(\mathbf{x}_1,\mathbf{x}_0)$ so that the resulting tree
is still a minimum spanning tree on $X\cup\{\mathbf{y}\}$.  From this, the statement of the lemma follows.

Assume first that the path in $T^{\,\prime}$ connecting $\mathbf{y}$ and $\mathbf{x}_0$ passes through $\mathbf{x}_1$. In this case we can
replace $(\mathbf{y},\mathbf{x}_1)$ by $(\mathbf{y},\mathbf{x}_0)$ without losing connectivity or increasing the length of $T^{\,\prime}$.

Assume, on the other hand, that the path in $T^{\,\prime}$ connecting $\mathbf{y}$ and $\mathbf{x}_0$ does not pass through $\mathbf{x}_1$.

\textbf{Claim:} $\|\mathbf{x}_0 - \mathbf{x}_1 \| \leq \|\mathbf{y} - \mathbf{x}_1\|$.\\
Without loss of generality,  we assume that $\mathbf{y}=\mathbf{o}$, $\|\mathbf{x}_1\|=1$, and $K(\mathbf{o},
\theta_i,\theta_{i+1})$ intersects the boundary of the unit ball $B$ in an arc from $\mathbf{a}$ to $\mathbf{b}$ (with $\|\mathbf{a} -
\mathbf{b}\| = 1$). We can also assume, without loss of generality, that $\mathbf{x}_1$ lies on the same side of the line through
$\mathbf{o}\mathbf{x}_0$ as $\mathbf{b}$.  The convexity of $B$ implies that the line segments $\mathbf{ox}_1$ and $\mathbf{ab}$ intersect, hence by
Lemma~\ref{tilemma} we have
\begin{equation}\label{eq1}
\|\mathbf{a} - \mathbf{x}_1 \| \leq \|\mathbf{a} - \mathbf{b}\| = 1.
\end{equation}
We now prove the claim via two cases, illustrated in Figure~\ref{figureCone}.
\begin{figure}[htb]
\begin{center}
\includegraphics[scale=0.5]{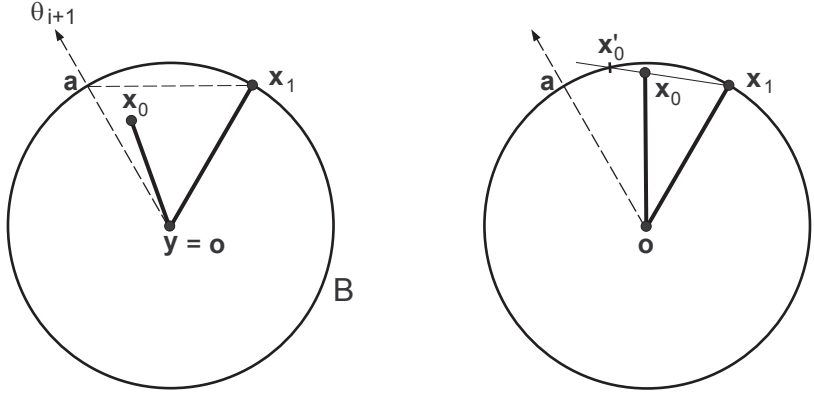}
\end{center}
\caption{The two cases of the Claim in the proof of Lemma~\ref{cony}\label{figureCone}}
\end{figure}
Firstly, suppose $\mathbf{x}_0$ and $\mathbf{o}$ are on the same side of $\mathbf{ax}_1$ (including the case where $\mathbf{x}_0$ lies on
$\mathbf{a}\mathbf{x}_1$). By Inequality~(\ref{eq1}) $\mathbf{a}$ lies in the unit ball centered at $\mathbf{x}_1$, so, by convexity, $\mathbf{x}_0$ also
lies in this unit ball. Hence,  $\| \mathbf{x}_0-\mathbf{x}_1\|\leq 1$ as required. For the second case, suppose that
$\mathbf{x}_0$ and $\mathbf{o}$ are on opposite sides of $\mathbf{ax}_1$. Let the ray from $\mathbf{x}_1$ passing through $\mathbf{x}_0$
intersect $B$ at $\mathbf{x}_0^\prime$. Then $\mathbf{x}_0^\prime$ and $\mathbf{o}$ are also on opposite sides of $\mathbf{ax}_1$, and hence, by
Lemma~\ref{tilemma}, $\| \mathbf{x}_0-\mathbf{x}_1\|\leq \| \mathbf{x}_0^\prime-\mathbf{x}_1\|\leq \|
\mathbf{a}-\mathbf{x}_1\|$. Therefore, $\| \mathbf{x}_0-\mathbf{x}_1\|\leq 1$ by Inequality~(\ref{eq1}), and the claim is
proven.

By the claim we can now replace the edge $(\mathbf{y},\mathbf{x}_1)$ by $(\mathbf{x}_1,\mathbf{x}_0)$ without losing connectivity or increasing
the length of $T^{\,\prime}$.
 \end{proof}

For each $i=0,...,5$, the $i$-th \textit{oriented Dirichlet cell} (ODC) of $\mathbf{w} \in X$ is the set:
$$\{\mathbf{y}\in\mathbb{R}^2:\|\mathbf{w}-\mathbf{y}\|=\mathrm{min}\{\| \mathbf{x}-\mathbf{y}\|:\mathbf{x}\in X
\cap K(\mathbf{y},\theta_i,\theta_{i+1})\}\}$$ In other words, this is the set of all points $\{\mathbf{y}\}$ whose closest terminal in the cone
$K(\mathbf{y},\theta_i,\theta_{i+1})$ is $\mathbf{w}$. We will show that the set of $i$-th ODCs, called the \textit{$i$-th ODC partition of $X$}
is a type of Voronoi diagram.

In \cite{chew} Chew and Drysdale present an ``expanding waves" view of Voronoi diagrams. If $n$ pebbles are dropped simultaneously into a pond,
the places where wave fronts meet define the Voronoi diagram on the $n$ points of impact. In the Euclidean case the wavefronts are circular, but
in theory any closed convex curve $C$ containing the origin can qualify as a wavefront and thereby define an abstract Voronoi diagram. For any such $C$ and
set of terminals $X$ we say that the resulting diagram is the \textit{Voronoi diagram of $X$ based on $C$}. The \textit{bisector based on} $C$ for any two points $\mathbf{x}, \mathbf{y}$ is defined as the intersection $V_{\mathbf{x}}\cap V_{\mathbf{y}}$ where $\{V_{\mathbf{x}},V_{\mathbf{y}}\}$ is the Voronoi diagram of $\{\mathbf{x},\mathbf{y}\}$ based on $C$.

We may define this Voronoi diagram more formally as follows. Let $\delta_C:\mathbb{R}^2\to\mathbb{R}$ be the distance function based on $C$; in
other words, for any points $\mathbf{x}_0,\mathbf{x}_1\in\mathbb{R}^2$ we let
$\delta_C(\mathbf{x}_0,\mathbf{x}_1)=\inf\{\lambda^{-1}:\lambda(\mathbf{x}_1-\mathbf{x}_0)\in C\}$. We then define a region
$V_{\mathbf{x}}=\{\mathbf{y}\in \mathbb{R}: \delta_C(\mathbf{x},\mathbf{y})=\min\{\delta_C(\mathbf{x}^\prime,\mathbf{y}):\mathbf{x}^\prime\in X\}\}$ for each
$\mathbf{x}\in X$. The set $\{V_{\mathbf{x}}\}$ is the required Voronoi diagram based on $C$. In Figure \ref{figureHexBis} we give an example of what the boundary of wavefronts look like when $C$ is a regular hexagon.

\begin{figure}[htb]
\begin{center}
\includegraphics[scale=0.3]{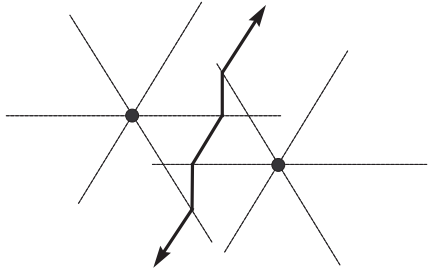}
\end{center}
\caption{Regular hexagon based Voronoi diagram for two points \label{figureHexBis}}
\end{figure}

\begin{proposition}For any $i=0,..,5$ the $i$-th ODC partition of $X$ is equal to the Voronoi diagram of $X$ based on the
sector $B \cap K(\mathbf{o}, 180^\circ+\theta_i,180^\circ+\theta_{i+1})$.
\end{proposition}
\begin{proof}This follows immediately from the central symmetry of $B$.
 \end{proof}

Figure \ref{figureODC} illustrates an ODC partition when the original unit ball is a circle (i.e. the Euclidean case) and therefore $C$ is a circular sector.

\begin{figure}[htb]
\begin{center}

\includegraphics[scale=0.3]{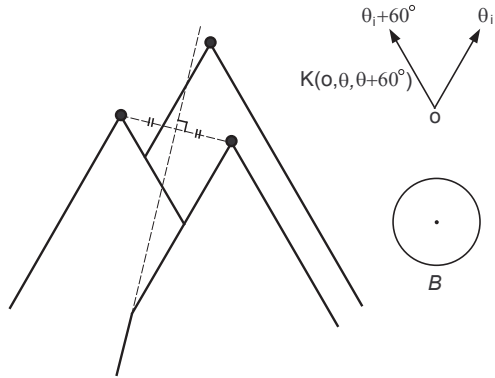}

\end{center}
\caption{An ODC partition of a three-terminal example \label{figureODC}}
\end{figure}

The next theorem now gives us the required time complexity for calculating the $i$-th ODC partition under certain conditions.

\begin{theorem}\label{conditions}\cite{chew} The Voronoi diagram of $n$ points based on a closed convex shape $C$ can be constructed in
$O(n\log n)$ time and $O(n)$ space as long as the following operations can be performed in constant time:
\begin{enumerate}
    \item Given two points, find the boundary where the two wavefronts meet.
    \item Given two such boundaries, compute their intersection(s).
\end{enumerate}
\end{theorem}

We therefore also impose the following restriction on $B$.

\noindent\textbf{Restriction 2} \textit{Let $C^{\, \prime}$ be any sector of $B$. Then, given any two points, we can find the boundary where the
two wavefronts based on $C^{\, \prime}$ meet, and, given two such boundaries, we can compute their intersection. Moreover, these operations can
be performed to within any fixed precision in constant time}.

The next step is to overlay the six $i$-th ODC partitions. The theorem we use, which is a result from \cite{georg}, assumes that the regions in each partition have boundaries consisting of straight line segments. We therefore state our third restriction on $B$.

\noindent\textbf{Restriction 3} \textit{The shape of $B$ implies that the $i$-th ODC partition of any set of points is piecewise linear (i.e. the boundary
of any ODC consists of straight line segments)}.

\begin{theorem}\cite{georg}\label{thr8}Let $q$ be any positive integer. Then $q$ linear plane partitions can be overlayed in $O(q^2n^2)$ time,
where $n$ is the total number of regions in each partition.
\end{theorem}

As a consequence of the previous results, within $O(n\log n)$ total time we can calculate the $i$-th ODC partition of the plane for
each $i$, and then, in a time of $O(n^2)$, overlay these six partitions resulting in the OODC partition. It is easily observed that
the OODC partition has $O(n^2)$ regions.

Let $R$ be a region of the OODC partition and let $\{D_j:j\in I\}$ (where $I$ is an index-set) be the set of ODCs such that $R=\bigcap\{D_j:j\in I\}$. Note then that $|I|\leq 6$. For each $j\in I$ suppose that $p_j$ is the terminal associated with $D_j$; in other words, $D_j$ is the ODC of $p_j$. Finally let $C_X(R)=\{p_j\}$. The power of the OODC partition lies in the next theorem, which now follows from Lemma
\ref{cony}.

\begin{theorem}\label{mainTh}Let $\mathbf{s}$ be any point in $R$. Then there exists a minimum spanning tree $T$ on $X\cup \{\mathbf{s}\}$ such that the set of
neighbours of $\mathbf{s}$ in $T$ is a subset of $C_X(R)$.
\end{theorem}

The question arises as to which norms exist with unit balls satisfying all three restrictions. Let $\mathcal{V}$ be the class of norms defined
by the condition that each norm's unit ball is either a polygon or an ellipse. Suppose that $\| \cdot
\|\in \mathcal{V}$ and that the corresponding unit ball is $B$. Clearly Restriction 1 is true for $B$. Given any two points, their
bisector (based on a sector of $B$) will be a polygonal line that can be computed with some simple vector operations (see \cite{chew}; for elliptic unit balls this follows since ellipses are linear transformations of a circle). The same
holds true for intersecting two boundaries, and therefore Restriction 2 is satisfied. Since any $i$-th ODC partition consists of segments of
bisectors and segments of the limiting rays of $K(\mathbf{y},\theta_i,\theta_{i+1})$ for some $\mathbf{y}$, Restriction 3 follows immediately. This class of norms includes (amongst others) the well-known Euclidean, rectilinear, $\ell_{\infty}$, and fixed-orientation planes. In this paper we do not undertake a deeper investigation into the question of whether Theorem \ref{thr8} can be generalised to include norms whose unit ball is neither linear nor elliptical, but leave it as an open question.

\section{Generalised Steiner Tree Problems for a Fixed Topology}\label{topology3}
By using the results of the previous section, specifically Theorem \ref{mainTh}, each main iteration of our algorithm produces a feasible internal topology which, recall, is a forest $\mathcal{F}$ spanning the set of all Steiner points such that $\mathcal{F}$'s internal nodes are Steiner points and its leaves are terminals. Finding the optimal coordinates of the Steiner points for the topology is a problem known in the literature
as the \textit{fixed topology Steiner tree problem}. We state the problem more formally as follows: given a set $A$ of $c\leq 6k^\prime$
embedded terminals, a set $\mathcal{S}$ of $k^\prime$ free (i.e. non-embedded) Steiner points, and a tree topology $\mathcal{T}$ spanning all
these nodes, we wish to find the coordinates of the Steiner points (i.e. find the set $S$) such that
$\alpha(\mathbf{e}_{\mathcal{T},A,S})$ is minimised, where $\mathbf{e}_{\mathcal{T},A,S}=(\| e_1\|,...,\|
e_{c+k^\prime-1}\|)$. Observe that we may assume $\mathcal{T}$ is a tree topology since each component of $\mathcal{F}$ may be solved separately.

The fixed topology problem is interesting in its own right, but is also a key step of our main algorithm. Since $k$ (and therefore $c$) is constant we are not particularly interested in the time complexity of this step. We therefore introduce one more restriction:

\noindent\textbf{Restriction 4} \textit{{$\alpha$ {and $B$} are such that a solution to the fixed topology Steiner tree problem is computable to within any fixed precision in finite time.}}

{As far as we know there are no instances of $\alpha$ and $B$ for which it has been demonstrated that the fixed topology problem is impossible to solve. Since we do not place restrictions on the methods or time-complexity of potential solutions to this problem (besides finiteness), we cannot fully characterise the class of cost-functions and norms that satisfy Restriction 4.} Note also that for many cost-functions and norms there may exist numerical methods (for instance gradient descent) that solve the fixed topology problem to any finite degree of accuracy. We now briefly discuss a few functions and norms that satisfy Restriction 4.

\begin{enumerate}
    \item[(1)]$\alpha(\mathbf{e}_{\mathcal{T},A,S})=\sum \| e_i \|$. In this case we are dealing with the well-known Steiner
    tree problem for a fixed topology. In the Euclidean plane the problem has an $O(c^2)$-time solution provided that no point has degree
    larger than $3$, see \cite{hwang}. Unfortunately, for the $k$-Steiner tree problem degree $4$ points do exist (but degree $5$ do not; see
    \cite{rubin}). A similar result holds for the rectilinear and other fixed orientation planes
    \cite{brazil}.
    \item[(2)] $\alpha(\mathbf{e}_{\mathcal{T},A,S})=\sum \| e_i\|^p$, $p>0$. This is referred to as the power-$p$ Steiner tree
    problem for a fixed topology. In the Euclidean plane with $p=2$, Ganley \cite{ganley} shows that the problem can be solved within time
    $O(c)$.
    \item[(3)] $\alpha(\mathbf{e}_{\mathcal{T},A,S})=\displaystyle\lim_{p\to\infty}\left(\sum\| e_i\|^p\right)^{1/p}$, i.e.
    the bottleneck Steiner problem for a fixed topology. This problem has an $O(c^2)$ solution
    in the rectilinear plane, see \cite{ganley2}. In the Euclidean and general $\ell_p$ planes there exists various numerical algorithms that can
    calculate a solution to any desired precision, see for instance \cite{drezner}, \cite{love}. A fully polynomial time approximation scheme (FPTAS)
    exists for the problem in the Euclidean plane (see \cite{ganley}). Recently Bae et al. \cite{bae2,bae1} produced the first exact algorithm for solving this problem.
\end{enumerate}

\section{Updating a Minimum Spanning Tree}\label{MST4}
This section deals with the final phase of our algorithm. At this stage the algorithm must select an appropriate set of cycle-edges to delete after forming the union of a minimum spanning tree on $X$ and a forest $F$, where $F$ has a given feasible internal topology and optimally located Steiner points for that topology. As in \cite{georg} for the case $k=1$, the fact that one can \textit{update} (in constant time) a minimum spanning tree on $X$ to include the Steiner points ultimately reduces time complexity: without an update method a minimum spanning tree would have to be constructed for every choice of feasible internal topology. By Corollary \ref{MST-coroll} in Section \ref{prelim1} we know that as long as the locations of the Steiner points are optimal then any minimum spanning tree on $X\cup S$ will also be an optimal generalised $k$-Steiner tree.

Many papers exist in the literature that deal with the time complexity of updating a minimum spanning tree when a single new node is introduced; see
for instance \cite{john} where the authors show that a tree on $n$ nodes can be updated with a new node in $O(\log n)$ parallel time using
$n/\log n$ exclusive read, exclusive write, parallel random access machines (EREW PRAMs). Georgakopoulos and Papadimitriou utilise a
preprocessing step in \cite{georg} so that a minimum spanning tree can be updated in constant time with a single new point.

Let $F$ be a solution to the fixed topology Steiner problem for some choice of feasible internal topology $\mathcal{F}$. As will become clear in Section~\ref{algorithm5} the requirement that the updated tree, say $T_F$, is a minimum spanning tree on its nodes can be slightly relaxed in our algorithm. It is only required that $T_F$ be a shortest total length tree spanning $X\cup S$ such that the neighbour-set of each Steiner point in $F$ is the same as in $T_F$. We therefore require that only edges not belonging to $F$ are deleted during the update process. The intuitive reason for modifying the update process in this way is to deal with cases when some component of $F$ is not a minimum spanning tree on its nodes (this can occur, for instance, in solutions to the bottleneck Steiner tree problem).

In the remainder of this section we introduce a few preliminary results, formalise the details of the update process, and prove in Theorems \ref{connected} and \ref{notconnected} that, given a preprocessing stage, updating only requires constant time.

Let $N(T,s)$ denote the set of neighbours of a node $s$ in a graph $T$. A forest $F$ with node-set $A \cup S$, where $A \subseteq
X$ and $S \subseteq \mathbb{R}^2$ with $\vert S\vert \leq k$, is called \textit{viable} if and only if $\{\mathbf{x}\in
V(F):\mbox{$\mathbf{x}$ is a leaf of $F$}\}=A$ and $\vert N(F,\mathbf{s})\vert\leq 6$ for every $\mathbf{s}\in S$. A shortest total length tree $T_F$, such that $V(T_F)= X\cup S$ and $N(T_F,\mathbf{s})=N(F,\mathbf{s})$ for every
$\mathbf{s}\in S$, is referred to as a \textit{minimum $F$-fixed spanning tree}. We use the symbol $P_T(\mathbf{x},\mathbf{y})$ to represent a
path through $T$ with endpoints $\mathbf{x}$ and $\mathbf{y}$, and we use $\ell_T(\mathbf{x},\mathbf{y})$ to denote the longest edge on
$P_T(\mathbf{x},\mathbf{y})$. We will make use of the following theorem.

\begin{theorem}\cite{korte}\label{kortee} A tree $T$ is a minimum spanning tree on $X$ if and only if for every
$\mathbf{x},\mathbf{y} \in X$, $\| e\| \leq\|\mathbf{x}-\mathbf{y}\|$ for every $e\in
E(P_T(\mathbf{x},\mathbf{y}))$.
\end{theorem}

Now let $T$ be a minimum spanning tree on $X$ and let PP1 denote a preprocessing stage to calculate $\ell_T(\mathbf{x},\mathbf{y})$ for every pair of nodes
$\mathbf{x},\mathbf{y} \in V(T)$. PP1 requires $O(n^2)$ time and $O(n^2)$ space. We incorporate a consistent tie-breaking procedure for choosing between edges of exactly the same length during PP1. The tie-breaking procedure places any order on $E(T)$ and chooses the earlier edge in this ordering
whenever a tie occurs. The next theorem shows that if $F$ is connected (i.e., $F$ is a tree) then updating a minimum spanning tree takes constant time.

\begin{theorem}\label{connected}Let $T$ be a minimum spanning tree on $X$, and assume that PP1 has been performed.
If $F$ is connected and viable, then a minimum $F$-fixed spanning tree $T_F$ can be constructed from $T$ in $O(k^2)$ time.
\end{theorem}

\begin{proof}
Let $G=T\cup F$, let $A= V(F) \cap X$, and note that $\vert A\vert \leq 6k$. A number of cycles may
occur in $G$, each one of them containing a path through $F$ with endpoints from $A$. Let $T^{\,\prime}$ be the graph obtained by deleting the
set of edges $\{\ell_T(\mathbf{x}_i,\mathbf{x}_j):\mathbf{x}_i,\mathbf{x}_j\in A, i\not= j \}$ from $G$. We will show that $T^{\,\prime} = T_F$,
which suffices to prove the proposition since $T^{\,\prime}$ can clearly be constructed in $O(k^2)$ time.

To prove that $T^{\,\prime} = T_F$ we first show that $T^{\,\prime}$ is acyclic and spans $X\cup S$. Every cycle of $G$ is of the form
$P_F(\mathbf{x}_i,\mathbf{x}_j),P_T(\mathbf{x}_j,\mathbf{x}_i)$, and therefore deleting every $\ell_T(\mathbf{x}_i,\mathbf{x}_j)$ from $G$
produces an acyclic graph. We use induction on $\vert A\vert$ to prove that $T^{\,\prime}$ is connected. Let
$A_b=\{\mathbf{x}_1,...,\mathbf{x}_b\}\subseteq A$ for some $b\in \{2,...,6k\}$, let $F_b$ be the subtree of $F$ induced by $S\cup A_b$, and let
$G_b=T\cup F_b$. Subtracting $L_b=\{\ell_T(\mathbf{x}_i,\mathbf{x}_j):1\leq i<j\leq b\}$ from $E(G_b)$ produces the graph $T_b$. For the base
case we let $b=2$. The only cycle of $G_2$ is  $P_F(\mathbf{x}_1,\mathbf{x}_2),P_T(\mathbf{x}_2,\mathbf{x}_1)$, and
$\ell_T(\mathbf{x}_2,\mathbf{x}_1)$ is an edge of this cycle. Therefore deleting $\ell_T(\mathbf{x}_2,\mathbf{x}_1)$ does not destroy the
connectivity of $T_2$ on $X\cup S$.

Next assume that $T_b$ spans $X\cup S$ for some $2\leq b\leq 6k-1$ and suppose that $\mathbf{x}_{b+1}\in X\backslash A_b$. Since $T_b$ is
connected and acyclic there is exactly one path connecting $\mathbf{x}_{b+1}$ to a node of $A_b$ not passing through any element of $S$, i.e. this path is of the form
$P_T(\mathbf{x}_{b+1},\mathbf{x}_r)$ for some unique $\mathbf{x}_r\in A_b$.  Let $A_{b+1}=A_b\cup \{\mathbf{x}_{b+1}\}$ and let $\mathbf{s}=N(F,\mathbf{x}_{b+1})$. Then $T_{b+1}$ is
the graph with $V(T_{b+1})= X\cup S$ and
$E(T_{b+1})=\left(E(T_b)\cup\{(\mathbf{s},\mathbf{x}_{b+1})\}\right)\backslash\{\ell_T(\mathbf{x}_{b+1},\mathbf{x}_i):\mathbf{x}_i\in A_b\}$.

\textbf{Claim:} For every $\mathbf{x}_i\in A_b$ either $\ell_T(\mathbf{x}_{b+1},\mathbf{x}_i)\in L_b$ or
$\|\ell_T(\mathbf{x}_{b+1},\mathbf{x}_i)\|=\|\ell_T(\mathbf{x}_{b+1},\mathbf{x}_r)\|$.\\
Let $\mathbf{x}_i\in A_b\backslash\{\mathbf{x}_r\}$ and consider the following two cases. If $\mathbf{x}_{b+1}$ lies on
$P_T(\mathbf{x}_i,\mathbf{x}_r)$ then
$\|\ell_T(\mathbf{x}_i,\mathbf{x}_r)\|=\max\{\|\ell_T(\mathbf{x}_i,\mathbf{x}_{b+1})\|,\|\ell_T(\mathbf{x}_{b+1},\mathbf{x}_r)\|\}
=\|\ell_T(\mathbf{x}_i,\mathbf{x}_{b+1})\|$ since $P_T(\mathbf{x}_{b+1},\mathbf{x}_r)$ is a path in $T_b$ and therefore does not contain
$\ell_T(\mathbf{x}_i,\mathbf{x}_r)$. Therefore $\ell_T(\mathbf{x}_{b+1},\mathbf{x}_i)\in L_b$. For the second case, if $\mathbf{x}_{b+1}$ does
not lie on $P_T(\mathbf{x}_i,\mathbf{x}_r)$ then let $\mathbf{y}$ be the first common point of the paths $P_T(\mathbf{x}_i,\mathbf{x}_r)$ and
$P_T(\mathbf{x}_{b+1},\mathbf{x}_r)$; see Figure~\ref{figureClaim}.
\begin{figure}[tb]
\begin{center}
\includegraphics[scale=0.5]{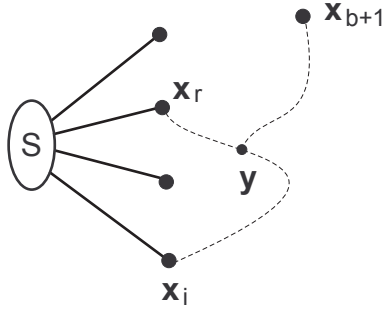}
\end{center}
\caption{The second case of the claim \label{figureClaim}}
\end{figure}
Note that $\mathbf{y}$ may be equal to $\mathbf{x}_r$. Clearly
\begin{equation}\label{eq2}
 \|\ell_T(\mathbf{x}_i,\mathbf{y})\|\geq\|\ell_T(\mathbf{y},\mathbf{x}_r)\|
\end{equation}
since $P_T(\mathbf{y},\mathbf{x}_r)$ is also a path in $T_b$. There are now two possibilities to consider; either
$\|\ell_T(\mathbf{x}_{b+1},\mathbf{x}_i)\|=\|\ell_T(\mathbf{x}_i,\mathbf{y})\|=\|\ell_T(\mathbf{x}_i,\mathbf{x}_r)\|\in L_b$, or
$\|\ell_T(\mathbf{x}_{b+1},\mathbf{x}_i)\|=\|\ell_T(\mathbf{x}_{b+1},\mathbf{y})\|=\|\ell_T(\mathbf{x}_{b+1},\mathbf{x}_r)\|$, where for each possibility
the second equality follows from Inequality~(\ref{eq2}). The claim follows.

By the above claim $E(T_{b+1})=\left(E(T_b)\cup\{(\mathbf{s},\mathbf{x}_{b+1})\}\right)\backslash\{\ell_T(\mathbf{x}_{b+1},\mathbf{x}_r)\}$ and
we deduce that $T_{b+1}$ has been constructed from $T_b$ by adding one edge from $F$ and then deleting an edge of $T$ on the resultant cycle. This completes the
induction argument, and hence $T^{\,\prime}$ is connected and spans $X\cup S$.

Next we prove that $T^{\,\prime}$ is a \textit{minimum} $F$-fixed spanning tree. Let $K$ be the complete graph on $X$. Furthermore, suppose that
the edges of $K$ are weighted by the function $w$, where
$$w((\mathbf{x},\mathbf{y})) = \left\{ \begin{array}{ll}
0 & \mbox{if $\mathbf{x}\in A$ and $\mathbf{y}\in A$},\\
\|\mathbf{x}-\mathbf{y}\| & \mbox{otherwise}. \end{array} \right.$$ Let $T_A$ be any spanning tree on $A$. Then clearly
$T^{\,\prime}$ is a minimum $F$-fixed spanning tree if and only if the graph $T_K$, where $V(T_K)=X$ and $E(T_K)=\left(E(T^{\,\prime})\cap
(X\times X)\right)\cup T_A$, is a minimum spanning tree of $K$ with the above weight function. But this follows from a simple application of
Theorem \ref{kortee}. Hence $T^{\,\prime} = T_F$, as required.
 \end{proof}

An immediate consequence of the above proof is the following result.

\begin{corollary}\label{corA}$\vert \{\ell_T(\mathbf{x}_i,\mathbf{x}_j):\mathbf{x}_i,\mathbf{x}_j\in A\}\vert=\vert A\vert-1$.
\end{corollary}

Next we extend Theorem~\ref{connected} to the case where $F$ is not necessarily connected, which must be considered if $k>1$. If $k>1$ we perform an additional preprocessing stage, PP2, to calculate a TRUE/FALSE table $H$ such that
$H_{e,\mathbf{y},\mathbf{z}}=\mbox{TRUE}$ if and only if edge $e \in E(P_T(\mathbf{y},\mathbf{z}))$. This requires at most $O(n^3)$ time and $O(n^3)$ space. For each connected component $F^i$ of $F$ let $A^i = V(F^i) \cap X$. We claim that \textit{Algorithm $F$-MST}, given in Table \ref{table1}, calculates a minimum $F$-fixed spanning tree for any viable forest $F$.

\begin{table}[h!b!p!]
\begin{center}
\begin{tabular}{|l|l|}
\hline
\multicolumn{2}{|c|}{\textbf{Algorithm $F$-MST}} \\
\hline
\multicolumn{2}{|c|}{\textbf{Input:} A set $X$ of points in the plane, a minimum spanning tree$\hspace{0.2cm}$} \\
\multicolumn{2}{|c|}{$T$ on $X$, and a viable forest $F$ with $t$ connected$\hspace{0.2cm}$} \\
\multicolumn{2}{|c|}{components.$\hspace{6.3cm}$} \\
\multicolumn{2}{|c|}{} \\
\multicolumn{2}{|c|}{\textbf{Output:} A minimum $F$-fixed spanning tree $T_F$ on $X\cup S$. $\hspace{0.6cm}$} \\
\hline
\textbf{Step} & \textbf{Description}\\
\hline
$\ $ & $\ $\\
$\ $ & Let $D^0=\emptyset$ and let $T^0=T$.\\
$\ $ & $\ $\\
1 & \textbf{For} $i=1$ \textbf{to} $t$ \textbf{Do}\\
$\ $ & \textbf{Begin}\\
1a & $\hspace{0.5cm}$\textbf{For} every distinct pair $\mathbf{x},\mathbf{y}\in A^i$ \textbf{Do}\\
$\ $ & $\hspace{0.5cm}$\textbf{Begin}\\
1a(i) & $\hspace{1cm}$Let $J$ be the graph with\\
$\ $ &$\hspace{1cm}V(J\,)=\left\{\begin{array}{ll}
\{\{\mathbf{x}\},\{\mathbf{y}\}\} & \mbox{if }i=1,\\
\{\{\mathbf{x}\},\{\mathbf{y}\},A^1,...,A^{i-1}\} & \mbox{if } i>1.
\end{array}\right.$\\
$\ $ & $\hspace{1cm}$and\\
$\ $ & $\hspace{1cm}E(J\,)=\{( U,U^{\,\prime}):\exists \mathbf{w}\in U \wedge\exists \mathbf{w}^\prime\in U^{\,\prime} \mbox{ such
that }$\\
$\ $ & $\hspace{2.6cm}H_{e',\mathbf{w},\mathbf{w}^\prime}=\mbox{FALSE}\ \forall e'\in D^{i-1}\}$.\\
$\ $ & $\ $\\
1a(ii) & $\hspace{1cm}$For every $e=(U,U^{\,\prime})\in E(J)$ let $\sigma_1^J(e)=\mathbf{w}$ and let\\
$\ $ & $\hspace{1cm}\sigma_2^J(e)=\mathbf{w}^\prime$ (where $\mathbf{w},\mathbf{w}^\prime$ are from the previous step).\\
1a(iii) & $\hspace{1cm}$Perform a search through $J$ to find the path\\
$\ $ & $\hspace{1cm}P^i=P_J(\{\mathbf{x}\},\{\mathbf{y}\})$.\\
$\ $ & $\hspace{1cm}$Let $\ell^{\,i}(\mathbf{x},\mathbf{y})$ be the edge from
$\left\{\ell_T(\sigma_1^J(e),\sigma_2^J(e)): e\in E(P^i)\right\}$\\
$\ $ & $\hspace{1cm}$of maximum length.\\
$\ $ & $\hspace{0.5cm}$\textbf{End}\\
$\ $ & $\ $\\
$\ $ & $\hspace{0.5cm}$Let $L^i=\{\ell^{\,i}(\mathbf{x},\mathbf{y}):\mathbf{x},\mathbf{y}\in A^i\}$,\\
$\ $ & $\hspace{1.15cm}D^i=D^{i-1}\cup L^i$,\\
$\ $ & $\hspace{1.15cm}G^i=T^{i-1}\cup F^i$,\\
$\ $ & $\hspace{1.15cm}T^i=\langle V(G^i),E(G^i)\backslash D^i\rangle$.\\
$\ $ & \textbf{End}\\
$\ $ & Let $T_F=T^t$.\\
\hline
\end{tabular}
\end{center}
\caption{Algorithm $F$-MST}\label{table1}
\end{table}

To understand Algorithm $F$-MST observe first that for $t=1$ the algorithm is identical to that of Theorem \ref{connected}. In other words $D^1=L^1=\{\ell_T(\mathbf{x}_i,\mathbf{x}_j):\mathbf{x}_i,\mathbf{x}_j\in A, i\not= j \}$ and this set is deleted from $G^1=T\cup F$ to produce $T_F$. The algorithm is inductive in nature, at each step adding component $F^i$ to the current tree $T^{i-1}$ and then deleting the longest edges on every cycle to get a tree $T^i$. All cycles that are obtained when adding $F^i$ to $T^{i-1}$ are of the form $P_{T^{i-1}}(\mathbf{x},\mathbf{y})\cup P_{F^i}(\mathbf{x},\mathbf{y})$ where $\mathbf{x},\mathbf{y}\in A^i$. Similarly to the proof of Theorem \ref{connected}, the required edge to be deleted at Step $i$ for the pair $\mathbf{x},\mathbf{y}$, namely $\ell^i(\mathbf{x},\mathbf{y})$, is simply the longest edge on $P_{T^{i-1}}(\mathbf{x},\mathbf{y})$. It is clear that $P_{T^{i-1}}(\mathbf{x},\mathbf{y})$ is either a path of $T$ or consists of alternating subpaths of $T$ and $F^j$ for various $j<i$. Since $k$ is constant it is possible to find, also in constant time, the subpaths of $P_{T^{i-1}}(\mathbf{x},\mathbf{y})$ that lie in $T$. By taking the maximum of the longest edges of all these paths in $T$ we get $\ell^i(\mathbf{x},\mathbf{y})$.

The purpose of $J$, as defined in the algorithm, is to have a graph of constant structural complexity that contains a representative edge for every path of $T^{i-1}$ that lies wholly in $T$. By specifying the nodes and edges of $J$ in the manner of Algorithm $F$-MST we are assured that $P_{T^{i-1}}(\mathbf{x},\mathbf{y})$ corresponds to a path in $J$ connecting $\{\mathbf{x}\}$ and $\{\mathbf{y}\}$. Observe that any path $P_{ab}$ in $T^{i-1}$ connecting nodes $\mathbf{x}_a\in A^a$ and $\mathbf{x}_b\in A^b$ lies entirely in $T$ if and only if $H_{e',\mathbf{x}_a,\mathbf{x}_b}=\mbox{FALSE}$ for all $e'\in D^{i-1}$, where $D^{i-1}$ is the set of edges that have been deleted from $T$ up to and including Step $i-1$. Given an edge $e$ of $J$ we need to know the endpoints of the path in $T$ represented by $e$. For this we introduce the functions $\sigma_j^J(e)$, $j=1,2$ in Algorithm $F$-MST.

\begin{figure}[htb]
\begin{center}
\includegraphics[scale=0.4]{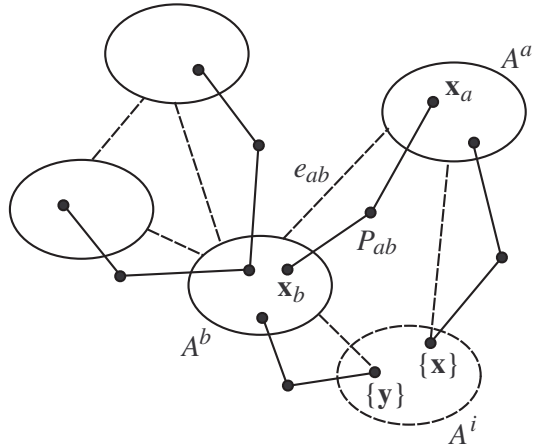}
\end{center}
\caption{An example of graph $J$ at Step $i$}\label{figEg1}
\end{figure}

Before formally proving the correctness of Algorithm $F$-MST we illustrate a few of the above concepts in Figure \ref{figEg1}. The solid-boundary ellipses and the sets $\{\mathbf{x}\},\{\mathbf{y}\}$ are nodes of $J$, the dashed lines are edges of $J$, and the solid lines and circles are edges and nodes of $T$. For $P_{ab}$ we have $\sigma_1^J(e_{ab})=\mathbf{x}_a$ and $\sigma_2^J(e_{ab})=\mathbf{x}_b$. Notice that $J$ is not necessarily a tree, and therefore it is not immediately clear that there will be a unique path in $J$ connecting $\{\mathbf{x}\}$ and $\{\mathbf{y}\}$. The next lemma settles this question. We assume by the inductive hypothesis that $T^{i-1}$ is a tree; the base case $i=1$ holds from Theorem \ref{connected}. Observe that for every path $P$ of $J$ connecting some $A^r$ and $A^d$ there exists a unique path $W(P)$ in $T^{i-1}$ connecting some pair of nodes $\mathbf{x}_r\in A^r$ and $\mathbf{x}_d\in A^d$.

\begin{figure}[htb]
\begin{center}
\includegraphics[scale=0.4]{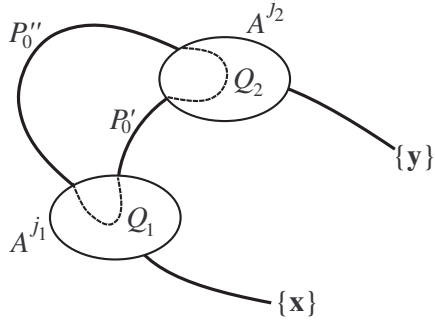}
\end{center}
\caption{Proof of Lemma \ref{JLem}}\label{figEg2}
\end{figure}

\begin{lemma}\label{JLem}$J$ contains at most one path connecting $\{\mathbf{x}\}$ and $\{\mathbf{y}\}$.
\end{lemma}
\begin{proof}
To get a contradiction let $P'$ and $P''$ be two distinct paths of $J$ connecting $\{\mathbf{x}\}$ and $\{\mathbf{y}\}$. Let $A^{j_1}$ and $A^{j_2}$ be distinct nodes of $P'$ such that the subpath $P_0''$ of $P''$ connecting $A^{j_1}$ and $A^{j_2}$ shares no internal nodes with $P'$; see Figure \ref{figEg2}. The pair $A^{j_1}$ and $A^{j_2}$ must exist since $P''$ is a proper path (i.e., no nodes are repeated in the $\{\mathbf{x}\}-\{\mathbf{y}\}$ walk through $P''$). Let $P_0'$ be the subpath of $P'$ connecting $A^{j_1}$ and $A^{j_2}$. For $c=1,2$ let $Q_c$ be the unique path through $F^{j_c}$ connecting the distinct endpoints of $W(P_0')$ and $W(P_0'')$ in $A^{j_c}$; if $W(P_0')$ and $W(P_0'')$ share the same endpoint in $A^{j_c}$ then $Q_c$ is the empty set. Clearly then $Q_1\cup W(P_0'')\cup Q_2\cup W(P_0')$ is a cycle of $T^{i-1}$, which contradicts the inductive hypothesis.
\end{proof}

\begin{corollary}\label{pathLem}Let $\mathbf{x},\mathbf{y}\in A^i$. Then $\ell^{\,i}(\mathbf{x},\mathbf{y})$, as defined in Algorithm $F$-MST, is the longest edge of $P_{T^{i-1}}(\mathbf{x},\mathbf{y})$, excluding any edges of $F$.
\end{corollary}
\begin{proof}
By the previous lemma there is a unique path $P^i=P_J(\{\mathbf{x}\},\{\mathbf{y}\})$ for Algorithm $F$-MST to find. The required longest edge on this path is the maximum of the longest edges for each subpath of $W(P^i)$ containing edges of $T$ only. The result follows.
\end{proof}

We now prove the main result of this section. The theorem implies that even in the case when $F$ is disconnected, our update method, as described in Algorithm $F$-MST, produces an optimal $F$-fixed spanning tree in constant time.

\begin{theorem}\label{notconnected}Let $T$ be a minimum spanning tree on $X$, and assume that preprocessing steps PP1 and PP2 have been performed. If $F$
is a viable forest then Algorithm $F$-MST correctly produces a minimum $F$-fixed spanning tree $T_F$ in at most $O(k^{2k+3}k!)$ time.
\end{theorem}
\begin{proof}
The proposition is verified by using induction on $t$ (the number of connected components of $F$). Theorem \ref{connected} proves the base case: $T^1$ is connected, acyclic, and a minimum $F^1$-fixed spanning tree. Similar reasoning is used to prove that each subsequent $T^i$ is connected and acyclic. At each inductive step, Corollary \ref{pathLem}
assures us that Algorithm $F$-MST correctly deletes the longest edge (excluding edges of $F$) of any new cycle formed. Let $\mathbf{F}_i=\displaystyle\bigcup_{j\leq i}F^j$. To prove minimality of
$T^i$ we once again construct (analogously to Theorem \ref{connected}) a weighted complete graph $K$ on $X$ and a tree $T_K$, such that
$T^i$ is a minimum $\mathbf{F}_i$-fixed spanning tree on $X\cup S$ if and only if $T_K$ is a minimum spanning tree of $K$. Theorem \ref{kortee} then
completes the minimality proof.

To verify the time complexity first note that $t\leq k$ and $\vert A^i\vert\leq 6k$. Line 1a of the algorithm requires $O(k^2)$ time,
Line 1a(i) requires at most $O(k^{2k}k!)$ time and Line 1a(iii) requires $O(k)$ time.
\end{proof}

\section{The Main Algorithm}\label{algorithm5}

\begin{table}[h!b!p!]
\begin{center}
\begin{tabular}{|l|l|l|}
\hline
\multicolumn{3}{|c|}{\textbf{Algorithm $k$-GSMT}} \\
\hline
\multicolumn{3}{|c|}{\textbf{Input:} A set $X$ of $n$ points in the plane, a unit ball $B$, a positive integer $k$,$\hspace{0.2cm}$} \\
\multicolumn{3}{|c|}{and a symmetric $\ell_1$-optimisable function $\alpha. \hspace{2.9cm}$} \\
\multicolumn{3}{|c|}{} \\
\multicolumn{3}{|c|}{\textbf{Output:} A set $S$ of at most $k$ Steiner points, and a tree $T^*$ interconnecting $\hspace{0.3cm}$} \\
\multicolumn{3}{|c|}{$X\cup S$, such that
$\| T^*\|_\alpha={\displaystyle \min_{\mathcal{T},S^\prime}\alpha(\mathbf{e}_{\mathcal{T},X,S^\prime})}.\hspace{2.9cm}$} \\
\hline
\textbf{Step} & \textbf{Description} & \textbf{Time} \\
\hline
$\ $ & $\ $ & $\ $ \\
1 & Construct the OODC partition of $X$. & $O(n\log n)$ \\
2 & Construct a minimum spanning tree $T$ on $X$. & $O(n\log n)$ \\
3 & Perform preprocessing steps PP1 and PP2 on $T$. & $O(n^3)$ \\
$\ $ & $\ $ & $\ $ \\
4 & \textbf{For} every $k^\prime\leq k$ and each choice (with repetition) & $O(n^{2k})$\\
$\ $ & of $k^\prime$ regions, $R_1,...,R_{k^\prime}$, of the OODC partition \textbf{Do} & $\ $ \\
$\ $ & \textbf{Begin} & $\ $ \\
4a(i) & $\hspace{1cm}$Associate the free Steiner point $s_i$ with region $R_i$. & $\ $\\
4a(ii) & $\hspace{1cm}$Let $\mathcal{G}$ be the graph consisting of the vertices & $\ $\\
$\ $ & $\hspace{1cm}\bigcup C_X(R_i)\cup \{s_1,...,s_{k^\prime}\}$, all edges $(s_i,s_j), i \neq j$, and & $\ $ \\
$\ $ & $\hspace{1cm}$all edges $(s_i,\mathbf{x})$ for every $\mathbf{x} \in C_X(R_i)$. & $\ $\\
4a(iii) & $\hspace{1cm}$Let $\mathcal{G}^*$ be the set of all viable subforests of $\mathcal{G}$. & $\ $\\
$\ $ & $\ $ & $\ $ \\
4b & $\hspace{1cm}$\textbf{For} each $\mathcal{F}\in\mathcal{G}^*$ \textbf{Do} & $O(f_1(k))$\\
$\ $ & $\hspace{1cm}$\textbf{Begin} & $\ $\\
4b(i) & $\hspace{1.5cm}$Solve the fixed topology generalised Steiner & $O(f_2(k))$ \\
$\ $ & $\hspace{1.5cm}$tree problem for $\mathcal{F}$ to get the forest $F$. & $\ $ \\
4b(ii) & $\hspace{1.5cm}$Run Algorithm $F$-MST with input $T$ and $F$, & $O(k^{2k+3}k!)$ \\
$\ $ & $\hspace{1.5cm}$and let $T_F$ be its output.& $\ $ \\
$\ $ & $\hspace{1cm}$\textbf{End} & $\ $\\
$\ $ & \textbf{End} & $\ $ \\
5 & Select a smallest total cost $T_F$ produced and let $T^*=T_F$. & $\ $\\
$\ $ & Let $S$ be the set of Steiner points of $T^*$. & $\ $\\
\hline
\end{tabular}
\end{center}
\caption{Algorithm $k$-GSMT}\label{table2}
\end{table}

We present \textit{Algorithm $k$-GSMT}, in Table \ref{table2}. As stated in Section 1, the algorithm contains three main phases for a given iteration. Lines 4a(i)-4a(iii) represent the first phase, namely the construction of a feasible internal topology. The set $\mathcal{G}^*$ contains all feasible internal topologies as specified by a given choice of regions of the OODC partition. Line 4b(i) performs the second phase by solving the fixed topology problem for the current feasible internal topology. Line 4b(ii) executes the minimum spanning tree update process, which is the final phase for the given iteration. To prove correctness of Algorithm $k$-GSMT we first need some definitions and two observations.

Let $T_{\mathrm{opt}}$ be a generalised $k$-Steiner minimum tree on $X$. Let $S$ be the set
of Steiner points in $T_{\mathrm{opt}}$ and let $F_{\mathrm{opt}}$ be the subforest of $T_{\mathrm{opt}}$ induced by the edges of
$T_{\mathrm{opt}}$ incident with elements of $S$. Let $F_{\mathrm{opt}}^i$ be a connected component of $F_{\mathrm{opt}}$ with $k_i$ Steiner
points and terminal set $A^i \subseteq X$. Note that, like $T_{\mathrm{opt}}$, $F_{\mathrm{opt}}$ may not be unique for a given set $X$.

\begin{observation}\label{Fiopt}$F_{\mathrm{opt}}^i$ is a generalised $k_i$-Steiner minimum tree on $A^i$.
\end{observation}

\begin{observation}\label{Fiswap}Let $Y^i$ be any generalised $k_i$-Steiner minimum tree on $A^i$. Suppose
we transform $T_{\mathrm{opt}}$ to $T^{\,\prime}$ by replacing the subtree $F_{\mathrm{opt}}^i$ on $T$ by $Y^i$. Then $T^{\,\prime}$ is also a
generalised $k$-Steiner minimum tree on $X$.
\end{observation}

We can now prove correctness.

\begin{proposition}\label{correctness}
If $B$ and $\alpha$ satisfy Restrictions 1--4 then Algorithm $k$-GSMT constructs, in a time of $O(n^{2k})$, a tree $T^*$ that is a generalised $k$-Steiner minimum tree on the terminal set $X$.
\end{proposition}

\begin{proof}
By the properties of the OODC partition, during the course of the algorithm a forest $F$ induced by edges incident with Steiner points is
constructed with connected components $F^i$ such that each $F^i$ has the same terminal set (i.e. $A^i$) and the same topology as
$F_{\mathrm{opt}}^i$, and therefore
\begin{equation}\label{eqFi}
   \| F^i \|_{\alpha}=\| F_{\mathrm{opt}}^i \|_{\alpha},
\end{equation}
where, recall, for any $T$ the symbol $\|T\|_\alpha$ denotes the cost of $T$ with respect to $\alpha$. We now consider two cases.

Suppose, for the first case, that $F_{\mathrm{opt}}=F$. Step~4b(ii) of the algorithm constructs a minimum $F$-fixed spanning tree $T_F$ on
$X\cup S$. Since $T_{\mathrm{opt}}$ is a minimum spanning tree on $X\cup S$ and contains $F$ as a subforest, it follows that $T_F$ is also a
minimum spanning tree on $X\cup S$. By Corollary ~\ref{MST-coroll}, $T^*=T_F$ is a generalised $k$-Steiner minimum tree on $X$, as required.

If, on the other hand, $F_{\mathrm{opt}}\not=F$, then there is a tree $T_F$ constructed in Step~4b(ii) of the algorithm that is the same as in
the previous paragraph, except each $F^i$ is replaced by $F_{\mathrm{opt}}^i$. By Equation~(\ref{eqFi}) and Observation~\ref{Fiswap} it again follows
that $T^*=T_F$ is a generalised $k$-Steiner minimum tree on $X$.

By using Cayley's formula and the observation that each spanning forest is a subgraph of a spanning tree which has $k-1$ edges, we get an upper bound for $f_1(k)$ of $126^k.\,k^{k-2}$ in Step 4b of the algorithm. The function $f_2$ in Step 4b(i) will depend on the relevant generalised $k$-Steiner tree problem and will be a function of $k$ only. Since $k$ is constant the overall time complexity of Algorithm $k$-GSMT is $O(n^{2k})$. Note that if $k=1$ or if the there is only one component, then the algorithm takes $O(n^2)$ time since we do not run PP2.
\end{proof}

\begin{theorem}{For any planar norm and symmetric $\ell_1$-optimisable cost function satisfying Restrictions 1-4 there exists a polynomial time algorithm with complexity $O(n^{2k})$ that solves the generalised $k$-Steiner tree problem for constant $k$.}
\end{theorem}
\section{Conclusion}
The outcome of this paper is a generalisation, on multiple fronts, of Georgakopoulos and Papadimitriou's $O(n^2)$ solution to the
$1$-Steiner tree problem. By utilising abstract Voronoi diagrams, we build on their complexity-reducing concept of oriented Dirichlet cell
partitions. The result is a broadening of the scope of these partitions to include terminal sets in a larger class of normed planes. A bigger challenge
in our research was to construct a generalisation to $k$ Steiner points. We achieve this by producing a novel method of updating a minimum
spanning tree to include a fixed subtree. A two-part preprocessing stage allows this to be done in constant time with respect to the total
number of terminals. One of the key observations of our research was that the main algorithm from \cite{georg} basically pertains to any
``Steiner-like" problem with cost function $\alpha$, as long as it is guaranteed that a solution exists which is optimal with respect to
$\alpha$ and is also a minimum spanning tree on its complete set of nodes. This fact allows us to accommodate the class of generalised
$k$-Steiner tree problems with symmetric $\ell_1$-optimisable cost functions. The result is an $O(n^{2k})$-time
solution to this class of problems.

It may be possible to generalise our algorithm to higher dimensional spaces, at least in the Euclidean case. A natural starting point could be
Monma and Suri's paper \cite{monma}, where a partition of $d$-dimensional Euclidean space is constructed that has similar topology-limiting
properties as the oriented Dirichlet cell partition.

\end{document}